\documentclass[12pt]{amsart}
\usepackage{calc,amssymb,amsthm,amsmath, ulem, fullpage}
\usepackage{alltt}

\RequirePackage[dvipsnames,usenames]{color}  

\input xy
\xyoption{all}

\DeclareFontFamily{OMS}{rsfs}{\skewchar\font'60}
\DeclareFontShape{OMS}{rsfs}{m}{n}{<-5>rsfs5 <5-7>rsfs7 <7->rsfs10 }{}
\DeclareSymbolFont{rsfs}{OMS}{rsfs}{m}{n}
\DeclareSymbolFontAlphabet{\scr}{rsfs}

\newtheorem{theorem}{Theorem}[section]
\newtheorem{lemma}[theorem]{Lemma}
\newtheorem{proposition}[theorem]{Proposition}
\newtheorem{corollary}[theorem]{Corollary}

\theoremstyle{definition}
\newtheorem{definition}[theorem]{Definition}

\theoremstyle{remark}
\newtheorem{remark}[theorem]{Remark}

\newtheorem{question}[theorem]{Question}

\DeclareMathOperator{\Spec}{{Spec}}
\newcommand{\ba}{\mathfrak{a}}

\newcommand{\bm}{\mathfrak{m}}
\newcommand{\bN}{\mathbb{N}}
\newcommand{\bZ}{\mathbb{Z}}
\DeclareMathOperator{\Hom}{Hom}

\newcommand{\blank}{\underline{\hskip 10pt}}
\newcommand{\bQ}{\mathbb{Q}}
\DeclareMathOperator{\Image}{Image}

\DeclareMathOperator{\Supp}{{Supp}}
\newcommand{\tensor}{\otimes}
\newcommand{\tld}{\widetilde }
\newcommand{\bR}{\mathbb{R}}

\theoremstyle{definition}
\newtheorem{definition-proposition}[theorem]{Definition-Proposition}

\renewcommand{\to}[1][]{\xrightarrow{\ #1\ }}

\newcommand{\Frob}[2]{{#1}^{1/p^{#2}}}
\newcommand{\Frobp}[2]{{(#1)}^{1/p^{#2}}}

\renewcommand{\O}{\scr O}

\begin{document}

\title{A refinement of sharply $F$-pure and strongly $F$-regular pairs}
\author{Karl Schwede}

\address{Department of Mathematics\\ University of Michigan\\ East Hall
530 Church Street \\ Ann Arbor, Michigan, 48109}
\email{kschwede@umich.edu}


\thanks{The author was partially supported by a National Science Foundation postdoctoral fellowship.}

\begin{abstract}
We point out that the usual argument used to prove that $R$ is strongly $F$-regular if and only if $R_{Q}$ is strongly $F$-regular for every prime ideal $Q \in \Spec R$, does not generalize to the case of pairs $(R, \ba^t)$.  The author's definition of sharp $F$-purity for pairs $(R, \ba^t)$ suffers from the same defect.  We therefore propose different definitions of sharply $F$-pure and strongly $F$-regular pairs.  Our new definitions agree with the old definitions in several common contexts, including the case that $R$ is a local ring.

\end{abstract}
\maketitle
\numberwithin{equation}{theorem}
\section{Introduction}

The notion of a strongly $F$-regular ring was introduced by Hochster and Huneke in \cite{HochsterHunekeFRegularityTestElementsBaseChange} because it was easily seen to be well behaved with respect to localization (this is in contrast to weak $F$-regularity).  It later was discovered that strongly $F$-regular rings (in characteristic $p > 0$) were closely related to rings with Kawamata log terminal singularities (in characteristic $0$), see \cite{HaraWatanabeFRegFPure}, \cite{HaraRatImpliesFRat} and \cite{TakagiInterpretationOfMultiplierIdeals}.  However, the notion of Kawamata log terminal singularities extends to pairs $(R, \ba^t)$ where $\ba \subseteq R$ is an ideal and $t > 0$ is a real number.  Therefore, it was natural to ask whether there is an analogous notion of strong $F$-regularity for pairs $(R, \ba^t)$.

In \cite{TakagiInversion}, Takagi gave such a definition and proved that it satisfied many properties similar to Kawamata log terminal singularities (also see \cite{TakagiWatanabeFPureThresh} and \cite{HaraWatanabeFRegFPure}).  In fact, by using this characteristic $p > 0$ definition, Takagi was able to prove remarkable results in characteristic zero for which there are still no known characteristic zero proofs, see for example \cite[Theorem 4.1]{TakagiInversion}.   We now state this definition:
\begin{definition}
\label{DefnOldStrongFReg}
 Suppose that $R$ is an $F$-finite reduced ring in characteristic $p > 0$, $\ba \subseteq R$ is an ideal and $t > 0$ is a real number.  Then we say that the pair $(R, \ba^t)$ is \emph{strongly $F$-regular} if, for every $d \in R^{\circ}$, there exists an integer $e > 0$ and an element $a \in \ba^{\lceil t(p^e - 1) \rceil}$ such that the inclusion $
 \xymatrix@R=0pt{
R \ar@{^{(}->}[r] & R^{1/p^e}, \text{defined by }
  1 \ar@{|->}[r] & (da)^{1/p^e}
 }$
  splits as a map of $R$-modules.
\end{definition}
\hskip -12pt The author of this note also defined a notion, for pairs, which he called sharp $F$-purity $(R, \ba^t)$, see Definition
\ref{DefnOldSharplyFPure}.  The reader should also compare with the notion of $F$-purity for pairs as defined in \cite{TakagiInversion}, \cite{HaraWatanabeFRegFPure} and \cite{TakagiWatanabeFPureThresh}.  Roughly speaking, $(R, \ba^t)$ is sharply $F$-pure if it satisfies the condition used to define strongly $F$-regular pairs in the case that $d = 1$; see \cite{SchwedeSharpTestElements} for details.

Takagi's definition of strongly $F$-regular pairs and the author's definition sharply $F$-pure pairs both work extremely well in the case that $R$ is a local ring.  Furthermore, strongly $F$-regular pairs have been studied largely in that context.
However, there are certain ways in which both definitions are unsatisfactory in the case that $R$ is a non-local ring.  For example, the author expects that $(R, \ba^t)$ being strongly $F$-regular (respectively, sharply $F$-pure) is a different condition than the localized pair $(R_{\bm}, \ba_{\bm}^t)$ being strongly $F$-regular (respectively sharply $F$-pure) for every maximal ideal $\bm$ of $\Spec R$.  On the other hand, in the classical non-pair setting, $R$ is strongly $F$-regular if and only if $R_{\bm}$ is strongly $F$-regular for every maximal ideal of $\Spec R$.  Note that Hara and Watanabe's definition of strong $F$-regularity, see \cite{HaraWatanabeFRegFPure}, does not suffer from this issue.\footnote{However, I do not know if Hara and Watanabe's definition of an $F$-pure pair $(R, \Delta)$ can be characterized locally.  The obstacle is a different one than we discuss below.}

Therefore, the main purpose of this paper is to state a refined definition of strong $F$-regularity and sharp $F$-purity for pairs, which satisfies the above localization criterion.  Our new definition for strong $F$-regularity is stated below.

\begin{definition}
\label{DefnNewStrongFReg}
 Suppose that $R$ is an $F$-finite reduced ring in characteristic $p > 0$, $\ba \subseteq R$ is an ideal and $t > 0$ is a real number.
 Then we say that the pair $(R, \ba^t)$ is \emph{locally strongly $F$-regular} if, for every $d \in R^{\circ}$, there exists an $e > 0$, and a map
\[
\phi \in \Hom_R(R^{1/p^e}, R) \cdot \left( d \ba^{\lceil t(p^e - 1) \rceil} \right)^{1/p^e}
\]
such that $\phi(1) = 1$ (or equivalently, that $\phi$ is surjective).
\end{definition}

The point is that the $\phi$ in Definition \ref{DefnNewStrongFReg} might be equal to a sum
\[
\phi(\blank) = \sum \phi_i(d a_i^{1/p^e} \blank)
\]
for $\phi_i \in \Hom_R(R^{1/p^e}, R)$ and $a_i \in \ba^{\lceil t(p^e - 1) \rceil}$, whereas in Definition \ref{DefnOldStrongFReg}, one would only consider sums with a single term.  In particular, a ``strongly $F$-regular'' pair is ``locally strongly $F$-regular''. We also state a better version of sharp $F$-purity for pairs, see on \ref{DefnBetterDefn}.  In fact, we state these definitions in greater generality: we state them for triples $(R, \Delta, \ba^t)$ where $\Delta$ is an effective $\bQ$-divisor on $X = \Spec R$.

For a pair $(R, \ba^t)$, Definition \ref{DefnOldStrongFReg} and Definition \ref{DefnNewStrongFReg} are equivalent under any of the following conditions (likewise for sharply $F$-pure pairs):
\begin{itemize}
 \item[(i)]  $R$ is a local ring, or
  \item[(ii)]  $R$ is an $\bN$-graded ring, $\ba$ is a graded ideal and $\Delta = 0$ or
 \item[(iii)]  $\ba$ is a principal ideal, or
 \item[(iv)]  $\Hom_R(R^{1/p^e}, R)$ is a free $R^{1/p^e}$-module for some $e$ greater than zero.  (This occurs, for example, if $R$ is sufficiently local and $\bQ$-Gorenstein with index not divisible by $p$).
\end{itemize}
It follows from (i) that Definition \ref{DefnNewStrongFReg} is equivalent to requiring that Definition \ref{DefnOldStrongFReg} holds after localizing at every maximal ideal, see Corollary \ref{CorNewAndOldAreRelated}.

This note also corrects a minor misstatement in the author's paper, \cite[Corollary 4.6]{SchwedeCentersOfFPurity}, where the author assumed that strong $F$-regularity for pairs was characterized locally.  See Remark \ref{RemErrorCorrection} for details.

Throughout this paper, all rings will be assumed to be commutative with unity, Noetherian, and contain a field of characteristic $p > 0$.  Furthermore, all rings will be assumed to be reduced and $F$-finite.

\section{Why Definition \ref{DefnOldStrongFReg} does not seem to localize well}

We first begin by reminding the reader of why Hochster and Huneke's original definition localizes well.  It is easy to see that if $R$ is strongly $F$-regular, then $R_{Q}$ is strongly $F$-regular for every $Q \in \Spec R$.  This direction also holds for pairs $(R, \ba^t)$.  Therefore, we will sketch the converse in the classical non-pair setting.

For any $d \in R^{\circ}$ and for each $e \geq 0$, consider the map
\[
 \Phi_{d, e} : \Hom_R(R^{1/p^e}, R) \rightarrow R
\]
which is the evaluation map at $d^{1/p^e}$ (that is, $\phi \mapsto \phi(d^{1/p^e})$).  It is easy to see that $R$ is strongly $F$-regular if and only if for every $d$, $\Phi_{d,e}$ is surjective for some $e > 0$. Since $R$ is $F$-finite, this is equivalent to requiring that $\left(\Phi_{d, e}\right)_{\bm}$ is surjective after localization at each maximal ideal $\bm \in \Spec R$.

Therefore, the only question is whether we can find a common $e$ so that the statement holds after localization at each maximal ideal $\bm \subset R$.  To this end, observe that if $(\Phi_{d, {e_0}})_{\bm}$ is surjective, it is also surjective for all $e > e_0$ (since a strongly $F$-regular local ring is also $F$-pure).

Now, as $e$ increases, the support of the modules $R / \Image(\Phi_{d, e})$ (which is also well behaved with respect to localization), is a decreasing set of closed subsets of $\Spec R$.  On the other hand, each point of $\Spec R$ is not contained in $\Supp \left(R / \Image(\Phi_{d,e}) \right)$ for $e$ sufficiently large.  Thus, we must have that $\Supp \left(R / \Image(\Phi_{d,e}) \right) = \emptyset$ for $e \gg 0$ since $R$ is Noetherian.  This implies that $R  = \Image(\Phi_{d,e})$ for $e \gg 0$.

Consider now a pair $(R, \ba^t)$.  Let us try to argue in the same way we did for the original definition of a strongly $F$-regular ring, see Definition \ref{DefnOldStrongFReg}.  In that case, we are restricting the map $\Phi_{d, e}$ to the set $S$ of maps $\phi : R^{1/p^e} \to R$ that can be written in the form $\phi(\blank) = \psi(a^{1/p^e} \blank)$ for some $\psi \in \Hom_R(R^{1/p^e}, R)$ and some $a \in \ba^{\lceil t(p^e - 1) \rceil}$.  The problem is, the set $S$ is not necessarily an $R$-submodule (or even a subgroup) of $\Hom_R(R^{1/p^e}, R)$.  Thus, we cannot say that $\Phi_{d,e}|_S$ is surjective if and only if $\Phi_{d,e}|_S$ is surjective after localizing at every maximal ideal.

\begin{remark}
\label{RemTheMapsAtLeastGenerate}
 While the set of the maps $\phi$ of the form $\phi(\blank) = \psi(a^{1/p^e} \blank)$ are not necessarily a $R^{1/p^e}$-submodule of $\Hom_R(R^{1/p^e}, R)$, they do generate the submodule $\Hom_R(R^{1/p^e}, R) \cdot \left( \ba^{\lceil t(p^e - 1) \rceil} \right)^{1/p^e}$.  This will be useful later.
\end{remark}

We also recall the author's original definition of sharply $F$-pure pairs.

\begin{definition} \cite{SchwedeSharpTestElements}
\label{DefnOldSharplyFPure}
 Suppose that $R$ is an $F$-finite reduced ring in characteristic $p > 0$, $\ba \subseteq R$ is an ideal and $t > 0$ is a real number.  Then we say that the pair $(R, \ba^t)$ is \emph{sharply $F$-pure} if there exists an $e > 0$ and an element $a \in \ba^{\lceil t(p^e - 1) \rceil}$, such that the inclusion $
 \xymatrix@R=0pt{
R \ar@{^{(}->}[r] & R^{1/p^e}}$ which sends $\xymatrix{
  1 \ar@{|->}[r] & (da)^{1/p^e}
 }$  splits as a map of $R$-modules.
\end{definition}

It is easy to see that this definition suffers from the same defect that Definition \ref{DefnOldStrongFReg} suffers from.

\section{A ``better'' definition}

Before we give our refined definition, we first fix some notation.

\begin{definition}
 A triple $(R, \Delta, \ba^t)$ is the combined information of
\begin{itemize}
 \item[(1)] an $F$-finite reduced ring $R$,
 \item[(2)] an ideal $\ba \subseteq R$ such that $\ba \cap R^{\circ} \neq \emptyset$,
 \item[(3)] a real number $t > 0$.
\end{itemize}
Furthermore, if $R$ is a normal domain, we also consider
\begin{itemize}
 \item[(4)] $\Delta$ an effective $\bQ$-divisor on $X = \Spec R$.\footnote{If $R$ is not a normal domain, we assume $\Delta = 0$.}
\end{itemize}
\end{definition}
If $\ba = R$ (respectively, if $\Delta = 0$) then we call the triple $(R, \Delta, \ba^t)$ a pair and denote it by $(R, \Delta)$ (respectively, by $(R, \ba^t)$).
Note that if $R$ is strongly $F$-regular, then $R$ is normal, so condition (4) is not so restrictive.   On the other hand, little is lost in this paper if one always assumes that $\Delta = 0$.

Given an effective integral divisor $D$ on $X = \Spec R$, we use the notation $R(D)$ to denote the global sections of the $\O_X$-module $\O_X(D)$.  Also note that for any effective divisor $D$, there is a natural map $R \rightarrow R(D)$.  Therefore, we have natural maps
\[
\pi_{\Delta, e}:  \Hom_R((R(\lceil (p^e - 1)\Delta \rceil))^{1/p^e}, R) \rightarrow \Hom_R(R^{1/p^e}, R).
\]
These maps are always injective.  Of course, if $\Delta = 0$, then $\pi_{\Delta, e}$ is the identity.

The notation
\[
\Image\left( \pi_{\Delta,e} \right) \cdot \left(J^{1/p^e} \right) = \Image\left(\Hom_R((R(\lceil (p^e - 1)\Delta \rceil))^{1/p^e}, R)
\rightarrow \Hom_R(R^{1/p^e}, R) \right) \cdot \left(J^{1/p^e} \right)
\]
will be used to denote the $R^{1/p^e}$-submodule of $\Hom_R(R^{1/p^e}, R)$ obtained by multiplying the $R^{1/p^e}$-submodule
\[
\Image(\pi_{\Delta, e}) \subseteq \Hom_R(R^{1/p^e}, R)
\]
by the $R^{1/p^e}$-ideal $J^{1/p^e}$.  It is important to note that the elements of this new submodule are still elements of $\Hom_R(R^{1/p^e}, R)$.

\begin{definition}
\label{DefnBetterDefn}
 Suppose that $(R, \Delta, \ba^t)$ is a triple.
 \begin{itemize}
\item{}  We say that $(R, \Delta, \ba^t)$ is \emph{locally strongly $F$-regular} if, for every $d \in R^{\circ}$, there exists $e > 0$, and a map $\phi \in \Image\left(\pi_{\Delta,e} \right) \cdot \left(d \ba^{\lceil t(p^e - 1) \rceil} \right)^{1/p^e}$ where $\phi : R^{1/p^e} \rightarrow R$ is surjective.
\item{}  We say that $(R, \Delta, \ba^t)$ is \emph{locally sharply $F$-pure} if there exists $e > 0$, and a map $\phi \in \Image\left(\pi_{\Delta,e} \right) \cdot \left(\ba^{\lceil t(p^e - 1) \rceil} \right)^{1/p^e}$ where $\phi : R^{1/p^e} \rightarrow R$ is surjective.

 \end{itemize}
\end{definition}

We now state several equivalent definitions.

\begin{lemma}
\label{LemmaEquivDefns}
 Suppose that $(R, \Delta, \ba^t)$ is a triple.  Then the following are equivalent:
  \begin{itemize}
	\item[(a)]  The triple $(R, \Delta, \ba^t)$ is locally strongly $F$-regular.
        \item[(b)]  For every $d \in R^{\circ}$, there exists some $e > 0$ and some $\phi \in \Image\left(\pi_{\Delta,e} \right) \cdot \left(d \ba^{\lceil t(p^e - 1) \rceil} \right)^{1/p^e}$ that splits the natural map $R \rightarrow R^{1/p^e}$.
        \item[(c)]  For every $d \in R^{\circ}$, there exists some $e > 0$ and some $\phi \in \Image\left(\pi_{\Delta,e} \right) \cdot \left(\ba^{\lceil t(p^e - 1) \rceil} \right)^{1/p^e}$ such that $\phi(d^{1/p^e}) = 1$.
        \item[(d)]  The map,
	    \[
            	\Image\left(\pi_{\Delta,e} \right) \cdot \left(d \ba^{\lceil t(p^e - 1) \rceil} \right)^{1/p^e} \rightarrow R,
            \]
             which evaluates an element $\phi$ at $1$, is surjective for some $e > 0$.
    \end{itemize}
  Furthermore, the following are also equivalent:
\begin{itemize}
	\item[(a$'$)]  The triple $(R, \Delta, \ba^t)$ is locally sharply $F$-pure.
        \item[(b$'$)]  For some $e > 0$, there exists $\phi \in \Image\left(\pi_{\Delta,e} \right) \cdot \left( \ba^{\lceil t(p^e - 1) \rceil} \right)^{1/p^e}$ that splits the natural map $R \rightarrow R^{1/p^e}$.
        \item[(c$'$)]  For some $e > 0$, there exists $\phi \in \Image\left(\pi_{\Delta,e} \right) \cdot \left(\ba^{\lceil t(p^e - 1) \rceil} \right)^{1/p^e}$ such that $\phi(1) = 1$.
        \item[(d$'$)]  The map,
        \[
            \Image\left(\pi_{\Delta,e} \right) \cdot \left(\ba^{\lceil t(p^e - 1) \rceil} \right)^{1/p^e} \rightarrow R,
        \]
        which evaluates an element $\phi$ at $1$, is surjective for some $e > 0$.
    \end{itemize}
\end{lemma}
\begin{proof}
Note first that condition (b) certainly implies condition (a).  Conversely, if $\phi \in \Image\left(\pi_{\Delta,e} \right) \cdot \left(d \ba^{\lceil t(p^e - 1) \rceil} \right)^{1/p^e}$ is surjective, then there exists $x \in R$ such that $\phi(x^{1/p^e}) = 1$.  But then the map $(\phi \cdot x^{1/p^e})$ sends $1$ to $1$ and so condition (b) is satisfied.  We will leave the equivalence of (b), (c) and (d) to the reader as they are similarly straightforward.   The equivalence of (a$'$) through (d$'$) is essentially the same.
\end{proof}

In Section \ref{SectionCasesWhereTheyAgree}, we will prove that if $R$ is local, then $(R, \Delta, \ba^t)$ is locally strongly $F$-regular (respectively, locally sharply $F$-pure) if and only if the localized triple $(R_Q, \Delta|_{\Spec R_Q}, \ba_Q^t)$ is strongly $F$-regular (respectively, sharply $F$-pure) for every $Q \in \Spec R$.  This justifies the \emph{``locally'' strongly $F$-regular} terminology.
However, the author feels that it would be better if the word ``locally'' was removed from future work (but that Definition
\ref{DefnBetterDefn} was still used).  Regardless, in this note, because there are two definitions, we will use the word ``locally'' to distinguish the new version.

\begin{remark}
 Definition \ref{DefnBetterDefn} can easily be generalized by replacing $\ba^t$ with a graded system of ideals $\ba_{\bullet}$, see \cite{HaraACharacteristicPAnalogOfMultiplierIdealsAndApplications} and \cite{SchwedeCentersOfFPurity}.  We won't do this here however.
\end{remark}

\begin{question}
Is there an example of a pair $(R, \ba^t)$ that is locally strongly $F$-regular (respectively, locally sharply $F$-pure) but not strongly $F$-regular (respectively, sharply $F$-pure)?
\end{question}

It seems that such an example may be difficult to construct (as there would be infinitely many conditions to check).

\section{The ``better'' definition behaves well with respect to localization}
\label{SectionBetterDefnLocalizes}

In this section, we show that locally strongly $F$-regular (respectively, locally sharply $F$-pure) triples can be characterized locally.  First however, we need a lemma.

\begin{lemma}
\label{LemmaCompositionIsOk}
 Suppose that we have maps:
\[
\phi \in \Image\left(\pi_{\Delta,e} \right) \cdot \left(\ba^{\lceil t(p^e - 1) \rceil} \right)^{1/p^e} \text{, }
\psi \in \Image\left(\pi_{\Delta,d} \right) \cdot \left(\ba^{\lceil t(p^d - 1) \rceil} \right)^{1/p^d}.
\]
Then $\phi \circ (\psi^{1/p^e})$ is contained in
\[
 \Image\left(\pi_{\Delta,e+d}\right) \cdot \left(\ba^{\lceil t(p^{e+d} - 1) \rceil} \right)^{1/p^{d+e}}.
\]
\end{lemma}
\begin{proof}
 It is sufficient to check this for some $\phi$ of the form $\phi(\blank) = \phi'(x^{1/p^e} \blank)$ where $\phi' \in \Image(\pi_{\Delta,e})$ and $x \in \ba^{\lceil t(p^e - 1) \rceil}$.  This is because of two facts:
\begin{itemize}
 \item[(1)]  Every element of $\Image\left(\pi_{\Delta,e} \right) \cdot \left(\ba^{\lceil t(p^e - 1) \rceil} \right)^{1/p^e}$ is a sum of elements of the form $\phi_i(\blank) = \phi'_i(x_j^{1/p^e} \blank)$ for $\phi'_i \in \Image(\pi_{\Delta,e})$ and $x_j \in \ba^{\lceil t(p^e - 1) \rceil}$.
\item[(2)]  A composition of a sum of maps with a map (ie, $(\phi_1 + \phi_2) \circ \psi^{1/p^e}$) is a sum of compositions of maps (ie, $\phi_1 \circ \psi^{1/p^e} + \phi_2 \circ \psi^{1/p^e}$).
\end{itemize}
Likewise, we may assume that $\psi$ is of the form $\psi(\blank) = \psi'(y^{1/p^d} \blank)$ for some $\psi' \in \Image(\pi_{\Delta, d})$ and some $y \in \ba^{\lceil t(p^d - 1) \rceil}$.

Now, $\phi'\left(x^{1/p^e} \left( \psi' (y^{1/p^{d}} \blank) \right)^{1/p^e} \right) = \phi'\left( \psi'^{1/p^e} (x^{1/p^e} y^{1/p^{d + e}} \blank)\right)$.  But we have that
\[
\begin{split}
x^{1/p^e} y^{1/p^{d+e}} = (x^{p^d} y)^{1/p^{d+e}} \\
\in \left( \ba^{p^d \lceil t(p^e - 1) \rceil} \ba^{\lceil t(p^{d} - 1) \rceil}
\right)^{1/p^{d+e}} \\
\subseteq \left( \ba^{\lceil t(p^{d+e} - 1) \rceil} \right)^{1/p^{d+e}}.
\end{split}
\]
Therefore, it is sufficient to show that $\phi' \circ (\psi')^{1/p^e} \in \Image(\pi_{\Delta, d+e})$.

If $\Delta = 0$, we are done, so we may assume $\Delta \neq 0$ and that $R$ is a normal domain.  Therefore, it is sufficient to check this at height one primes of $R$ since the modules $\Image(\pi_{\Delta,e+d})$ and $\Hom_R(R^{1/p^{e+d}}, R)$ are rank 1 reflexive $R^{1/p^{e+d}}$-modules.  However, at a height one prime $Q \in \Spec R$, the pair $(R_Q, \Delta|_{\Spec R_Q})$ can be identified with a pair $(R_Q, (f)^{1/n})$ where $n \Delta$ is integral and $f$ is a local defining equation for $n \Delta$ at $Q$.  Then the argument follows as in the case above, also see
\cite[Proof of Lemma 2.5]{TakagiInterpretationOfMultiplierIdeals}.
\end{proof}

\begin{remark}
\label{RemarkComposingIsOk}
Suppose that $(R, \Delta, \ba^t)$ is locally sharply $F$-pure (respectively, locally strongly $F$-regular), due to the existence of some $\phi \in \Image\left(\pi_{\Delta,e} \right) \cdot \left(\ba^{\lceil t(p^e - 1) \rceil} \right)^{1/p^e}$ with $1 \in \phi(R^{1/p^e})$ (respectively, with $1 \in \phi(d^{1/p^e} R^{1/p^e}))$).  Lemma \ref{LemmaCompositionIsOk} implies that for every $n \geq 1$, we can find
\[
\phi_n \in \Image\left(\pi_{\Delta,ne} \right) \cdot \left(\ba^{\lceil t(p^{ne} - 1) \rceil} \right)^{1/p^{ne}}
\]
with $1 \in \phi_n(R^{1/p^{ne}})$ (respectively, with $1 \in \phi_n(d^{1/p^{ne}}R^{1/p^{ne}})$).
\end{remark}

\begin{theorem}
\label{TheoremDefnLocalizes}
 A triple $(R, \Delta, \ba^t)$ is locally strongly $F$-regular (respectively locally sharply $F$-pure) if and only if $(R_Q, \Delta|_{\Spec R_Q}, \ba_Q^t)$ is locally strongly $F$-regular (respectively, locally sharply $F$-pure) for every ideal $Q \in \Spec R$.
\end{theorem}
\begin{proof}
 This may be obvious to experts, but because the original definition seems to lack this property, we prove it carefully here.
First we note that the direction $(\Rightarrow)$ is straightforward and not substantially different from the classical non-pair setting.  Thus we only prove the $(\Leftarrow)$ direction.  Suppose that for each $Q \in \Spec R$, $(R_Q, \Delta|_{\Spec R_Q}, \ba_Q^t)$ is locally strongly $F$-regular (respectively, locally sharply $F$-pure).  Fix a $d \in R^{\circ}$ (or set $d = 1$, if one is checking the sharply $F$-pure case).  By Lemma \ref{LemmaEquivDefns}(d), we see that for each $Q \in \Spec R$, there exists an $e_Q > 0$ so that the map which evaluates at $1$,
 \[
   E_{e_Q, Q} : \Image\left(\pi_{\Delta,e_Q} \right)_Q \cdot \left(d \ba^{\lceil t(p^{e_Q} - 1) \rceil} \right)^{1/p^{e_Q}}_Q \rightarrow R_Q,
 \]
 is surjective.  But then for each $Q$, this holds in an affine neighborhood $U_Q$ of $Q$.  We can cover $X = \Spec R$ by a finite collection of such neighborhoods $U_1, \ldots, U_n$ with corresponding surjective evaluation maps $E_{e_i, U_i}$.  Of course, the particular $e_i$'s associated to each neighborhood may vary.  However, Lemma \ref{LemmaCompositionIsOk} implies that if $E_{e_i, U_i}$ is surjective, then so is $E_{n e_i, U_i}$ for every $n > 0$.  Thus, increasing the $e_i$ if necessary, we can find a common $e$ that works on all $U_i$.  But then we are done, since a map of $R$-modules is surjective if and only if the corresponding maps on a finite affine cover of $\Spec R$ are surjective.
\end{proof}

\section{Cases where the two definitions agree}
\label{SectionCasesWhereTheyAgree}

In this section, we prove that the two definitions agree in the cases (i) through (iv) mentioned in the introduction.
We first do conditions (iii) and (iv).

\begin{proposition}
Suppose that $(R, \Delta, \ba^t)$ is a triple.  Further suppose that either:
\begin{itemize}
\item[(a)] $\ba$ is a principal ideal, or
\item[(b)]  $\Image\left(\pi_{\Delta,e} \right)$ is a free $R^{1/p^e}$-module for some $e > 0$.
\end{itemize}
  Then $(R, \Delta, \ba^t)$ is strongly $F$-regular (respectively, sharply $F$-pure) if and only if it is locally strongly $F$-regular (respectively, locally sharply $F$-pure).
\end{proposition}
\begin{proof}
First assume we are in case (b).  We claim that if $\Image\left(\pi_{\Delta,e} \right)$ is cyclic as an $R^{1/p^e}$-module for some $e > 0$, then $\Image\left(\pi_{\Delta,ne} \right)$ is also cyclic as an $R^{1/p^{ne}}$-module for all $n > 0$.  In the case that $\Delta = 0$, this is essentially an exercise in applying the adjointness of $\tensor$ and $\Hom$, see \cite[Lemma 3.9]{SchwedeFAdjunction}.  In the case that $\Delta \neq 0$, $R$ is normal and so one can reduce to the 1 dimensional case and argue in essentially the same way, see \cite[Corollary 3.10]{SchwedeFAdjunction}.   Therefore, Remark \ref{RemarkComposingIsOk} allows us to assume that we can find $e > 0$ so that condition (b) holds and, for that same $e$, we may find a surjective map in $\phi \in \Image\left(\pi_{\Delta,e} \right) \cdot \left(d \ba^{\lceil t(p^e - 1) \rceil} \right)^{1/p^e}$.

Now, in either case (a) or (b), every element of
\[
\Image\left(\pi_{\Delta,e} \right) \cdot \left(d \ba^{\lceil t(p^e - 1) \rceil} \right)^{1/p^e}
\]
can be written as a map of the form $\phi((da)^{1/p^e} \blank)$ for some $\phi \in \Image\left(\pi_{\Delta,e} \right)$ and some $a \in \ba^{\lceil t(p^e - 1) \rceil}$.  In the sharply $F$-pure case, set $d = 1$.  The result then follows.
\end{proof}

We now note that the two definitions are the same in the case that $R$ is local.

\begin{proposition}
\label{PropSameInLocalCase}
Suppose that $(R, \Delta, \ba^t)$ is a triple.  Further suppose that $(R, \bm)$ is local.  Then $(R, \Delta, \ba^t)$ is strongly $F$-regular (respectively, sharply $F$-pure) if and only if it is locally strongly $F$-regular (respectively, locally sharply $F$-pure).
\end{proposition}
\begin{proof}
Note that elements of the form $\phi((da)^{1/p^e} \blank)$, for some $\phi \in \Image\left(\pi_{\Delta,e} \right)$ and some $a \in \ba^{\lceil t(p^e - 1) \rceil}$, generate $\Image\left(\pi_{\Delta,e} \right) \cdot \left(d \ba^{\lceil t(p^e - 1) \rceil} \right)^{1/p^e}$ even as an $R$-module.  Therefore, if all these elements are sent, by evaluation at $1$, into the maximal ideal $\bm$, then so are any of their linear combinations.  For the proof in the sharply $F$-pure case, set $d = 1$.
\end{proof}

Finally, we verify the graded case at least assuming that $\Delta = 0$.  One can do similar things when $\Delta \neq 0$ but the statements become more complicated.

\begin{proposition}
\label{PropSameInGradedCase}
Suppose that $(R, \ba^t)$ is a pair.  Further suppose that $R = \oplus_{i \geq 0} R_i$ is an $\bN$-graded ring and $\ba$ is a graded ideal.  Then $(R, \ba^t)$ is strongly $F$-regular (respectively, sharply $F$-pure) if and only if it is locally strongly $F$-regular (respectively, locally sharply $F$-pure).
\end{proposition}
\begin{proof}
Suppose first that $R$ is locally strongly $F$-regular.  It is sufficient to show that $(R, \ba^t)$ is strongly $F$-regular in the usual sense (the case of sharply $F$-pure rings is similar).   We view $\Frob{R}{e}$ as a $\bZ[1/p^e]$-graded $R$-module.  Using standard techniques related to strong $F$-regularity, it is not difficult to see that it is sufficient to verify the statements of Lemma \ref{LemmaEquivDefns} for homogenous $d \in R^{\circ}$.  Note that $\Hom_R(\Frob{R}{e}, R)$ is generated by graded (degree-shifting) homomorphisms since $R$ is $F$-finite.  In particular, the image of the natural map
\[
     \Image\left(\pi_{\Delta,e} \right) \cdot \left(d \ba^{\lceil t(p^e - 1) \rceil} \right)^{1/p^e} \rightarrow R,
\]
which evaluates an element $\phi$ at $1$, is a graded submodule of $R$.  In particular, the map is surjective if and only if the image is not contained in $R_+$.
One then argues exactly the same as in the local case.
\end{proof}

One should note that most of the work done with the previous definition of strongly $F$-regular pairs was in the local setting, see for example \cite{TakagiInversion}, \cite{TakagiWatanabeFPureThresh}, \cite{MustataTakagiWatanabeFThresholdsAndBernsteinSato}, and \cite{TakagiPLTAdjoint}.

\begin{corollary}
\label{CorNewAndOldAreRelated}
A triple $(R, \Delta, \ba^t)$ is locally strongly $F$-regular (respectively, locally sharply $F$-pure) if and only if $(R_{\bm}, \Delta|_{\Spec R_{\bm}}, \ba_{\bm}^t)$ is strongly $F$-regular (respectively, sharply $F$-pure) for every maximal $\bm \in \Spec R$.
\end{corollary}

We now recall the definition of the test ideal.

\begin{definition}
\label{DefnTestIdeal} \cite{TakagiInterpretationOfMultiplierIdeals}, \cite{HochsterFoundations}, \cite{SchwedeCentersOfFPurity}
Let $X = \Spec R$ be an $F$-finite normal integral affine scheme.  Further suppose that $\Delta$ is an effective $\bQ$-divisor on $X$, $\ba \neq (0)$ is an ideal of $R$ and $t \geq 0$ is a real number.  We define the \emph{big test ideal}\footnote{The big test ideal is often called the non-finitistic test ideal and is often denoted by $\tld \tau(R; \Delta, \ba^t)$.} $\tau_b(R; \Delta, \ba^t)$ of the triple $(R, \Delta, \ba^t)$ to be the unique smallest non-zero ideal $J$ of $R$ such that
\begin{equation}
\label{EqnFCompatibleIdeal}
 \phi \left(\Frobp{\ba^{\lceil t(p^e - 1) \rceil} J}{e} \right) \subseteq J
\end{equation}
for all $e \geq 0$ and all
\[
\phi \in \Image\Bigg(\Hom_R\left( \Frobp{R(\lceil (p^e - 1)\Delta \rceil)}{e}, R\right) \rightarrow \Hom_R(\Frob{R}{e}, R) \Bigg).
\]
This ideal always exists in the context described.
\end{definition}

\begin{remark}
\label{RemarkTestElementDefinition}
Assume that $0 \neq c \in \tau_b(R; \Delta, \ba^t)$ (in other words, $c$ is a \emph{big sharp test element}).  Then
\[
\tau_b(R; \Delta, \ba^t) = \sum_{e \geq 0} \sum_{\phi} \phi(\Frobp{c \ba^{\lceil t(p^e - 1) \rceil}}{e}),
\]
where the inner sum is over
\[
\phi \in \Image\Bigg(\Hom_R\left( \Frobp{R(\lceil (p^e - 1)\Delta \rceil)}{e}, R\right) \rightarrow \Hom_R(\Frob{R}{e}, R) \Bigg).
\]
It is clear that the sum on the right satisfies the condition from Equation \ref{EqnFCompatibleIdeal}.  It is also easy to see that the sum on the right is non-zero (consider the case where $e = 0$).  Thus the containment $\subseteq$ is clear.  But the containment $\supseteq$ is also easy since $cR \subseteq \tau_b(R; \Delta, \ba^t)$ and again using Equation \ref{EqnFCompatibleIdeal}.  Thus the statement is proven.
\end{remark}

For more discussion on the big test ideal in this context, see \cite[Section 3]{BlickleSchwedeTakagiZhangDiscAndRat} or \cite[Subsection 2.2]{SchwedeCentersOfFPurity} and compare with  \cite{HaraYoshidaGeneralizationOfTightClosure}, \cite{TakagiInterpretationOfMultiplierIdeals}, and \cite{LyubeznikSmithCommutationOfTestIdealWithLocalization}.

\begin{corollary}
\label{CorTestIdealCharOfLocally}
 The big test ideal  $\tau_b(R; \Delta, \ba^t)$ is equal to $R$ if and only if the triple $(R, \Delta, \ba^t)$ is locally strongly $F$-regular.
\end{corollary}
\begin{proof}
Since the formation of the big test ideal commutes with localization, see \cite[Section 3]{BlickleSchwedeTakagiZhangDiscAndRat} and \cite[Lemma 2.1]{HaraTakagiOnAGeneralizationOfTestIdeals}, it is sufficient to prove Corollary \ref{CorTestIdealCharOfLocally} at each maximal ideal.  Therefore, we can reduce to the case of a local ring $(R, \bm)$.  Then the proof is exactly the same as in \cite[Lemma 2.3]{TakagiInterpretationOfMultiplierIdeals}, or \cite[Proposition 2.1]{HaraFRegAndFPureGradedRings}.  We sketch another approach where we avoid Matlis duality and instead use a version of \cite[Lemma 2.1]{HaraTakagiOnAGeneralizationOfTestIdeals} generalized to triples $(R, \Delta, \ba^t)$.

 Now, still assuming $(R, \bm)$ is local, $R = \tau_b(R; \Delta, \ba^t)$ if and only if for each $d \in R^{\circ}$ we have
 \[
 1 \in \sum_{e\geq 0} \left( \sum_{\phi_e \in \Image(\pi_{\Delta, e}) } \phi_e\left( (d \ba^{\lceil t(p^e - 1) \rceil})^{1/p^e} \right)  \right),
 \]
see Remark \ref{RemarkTestElementDefinition}, \cite[Lemma 2.1]{HaraTakagiOnAGeneralizationOfTestIdeals}, \cite[Lemma 3.5]{TakagiPLTAdjoint} and \cite[Lemma 2.20]{SchwedeCentersOfFPurity}.  But then the statement is obvious since $1$ is in the sum if and only if there are terms in the sum not contained in $\bm$.
\end{proof}

\section{$F$-pure thresholds, test ideals, and uniformly $F$-compatible ideals}
\label{SectionThresholdsTestIdealsUniformlyFCompatibleIdeals}

In this section we discuss how these new notions of strong $F$-regularity and sharp $F$-purity fit into the existing theory.  We also correct a small error of the current author, in the paper \cite{SchwedeCentersOfFPurity}.

\subsection{The $F$-pure threshold}

Recall that the \emph{$F$-pure threshold} of a pair $(R, \ba)$, where $R$ is a reduced $F$-finite $F$-pure (not necessarily local) ring is defined to be
\[
\begin{split}
c(\ba) = \sup\{ s \in \bR_{\geq 0} | \text{ the pair $(R, \ba^s)$ is $F$-pure} \}\text{ } \\= \sup\{ s \in \bR_{\geq 0} | \text{ the pair $(R, \ba^s)$ is sharply $F$-pure} \},
\end{split}
\]
see \cite[Definition 2.1]{TakagiWatanabeFPureThresh} and \cite[Proposition 5.3]{SchwedeSharpTestElements}.  This definition was stated originally for non-local rings, but we expect that a better definition would require that $(R, \ba^s)$ is locally sharply $F$-pure.  Note, most results about the $F$-pure threshold were shown in the case that $R$ is local.  However, there is the following notable exception:

\begin{remark}
The rationality result for the $F$-pure threshold found in \cite[Theorem 3.1]{BlickleMustataSmithDiscretenessAndRationalityOfFThresholds} or \cite{BlickleSchwedeTakagiZhangDiscAndRat} (at least when $R$ is strongly $F$-regular), is a rationality result for the \emph{locally-$F$-pure threshold}.  Note that in \cite[Theorem 3.1]{BlickleMustataSmithDiscretenessAndRationalityOfFThresholds}, it is assumed that $R$ is regular, but it is not assumed that $\Hom_R(R^{1/p^e}, R)$ is free as an $R^{1/p^e}$-module (though it is locally free).
\end{remark}

\subsection{The test ideal of a locally sharply $F$-pure pair}

We turn our attention again to test ideals.  One nice fact about sharply $F$-pure pairs is that the associated generalized test ideal (of \cite{HaraYoshidaGeneralizationOfTightClosure}) is a radical ideal.  We now show directly that this also holds for local sharp $F$-purity.

\begin{proposition} \cite[Corollary 3.15]{SchwedeSharpTestElements}
If $(R, \ba^t)$ is locally sharply $F$-pure, then the test ideal $\tau(R, \ba^t)$ (as defined in \cite{HaraYoshidaGeneralizationOfTightClosure}) is a radical ideal.
\end{proposition}
\begin{proof}
The proof is identical to the proof of \cite[Corollary 3.15]{SchwedeSharpTestElements} once one has Lemma \ref{LemmaFrobeniusClosure}, which we prove below.
\end{proof}

First we recall the following definition.

\begin{definition}
Given an ideal $I \subseteq R$, we define the \emph{$\ba^t$-sharp Frobenius closure of $I$}, denoted $I^{F^{\sharp}\ba^t}$ as follows.  The ideal $I^{F^{\sharp} \ba^t}$ is defined to be the set of elements $z \in R$ such that $\ba^{\lceil t(p^e - 1) \rceil} z^{p^e} \subseteq I^{[p^e]}$ for all $ e \gg 0$.
\end{definition}

\begin{lemma} \cite[Remark 3.11]{SchwedeSharpTestElements}
\label{LemmaFrobeniusClosure}
If $(R, \ba^t)$ is locally sharply $F$-pure, then $I^{F^{\sharp} \ba^t} = I$ for all ideals $I$.
\end{lemma}
\begin{proof}
Suppose that $z \in I^{F^{\sharp} \ba^t}$.  Thus there exists $e_0 > 0$ such that $\ba^{\lceil t(p^e - 1) \rceil} z^{p^e} \subseteq I^{[p^e]}$ for all $ e \geq e_0$.
Since $(R, \ba^t)$ is locally sharply $F$-pure, there exists a $\phi \in \Hom_R(R^{1/p^e}, R) \cdot \left( \ba^{\lceil t(p^e - 1) \rceil}\right)^{1/p^{e}}$, for some $e \geq e_0$, such that $\phi(1) = 1$ (we can increase $e$ due to Lemma \ref{LemmaCompositionIsOk}).  We then write $\phi = \phi_1 \cdot a_1^{1/p^e} + \dots + \phi_m \cdot a_m^{1/p^e}$ for $\phi_i \in \Hom_R(R^{1/p^e}, R)$ and $a_i \in \ba^{\lceil t(p^e - 1) \rceil}$.  Then for that same $e \geq e_0$,
\[
\begin{split}
z = \phi((z^{p^e})^{1/p^e}) \\
= \sum_{i = 1}^m \phi_i\left((a_i z^{p^e})^{1/p^e}\right) \\
 \subseteq \sum_{i = 1}^m \phi_i\left( \left( \ba^{\lceil t(p^e - 1) \rceil} z^{p^e} \right)^{1/p^e}\right) \\
 \subseteq \sum_{i = 1}^m \phi_i\left( \left(I^{[p^e]} \right)^{1/p^e} \right) \\
 \subseteq I.
\end{split}
\]
\end{proof}

\subsection{Uniformly $F$-compatible ideals and centers of $F$-purity}

We begin by recalling the definition of a uniformly $F$-compatible ideal.

\begin{definition}
Suppose that $(R, \Delta, \ba^t)$ is a triple.  Recall that an ideal $J \subseteq R$ is called \emph{uniformly $(\Delta, \ba^t, F)$-compatible} if for all $\phi \in \Image\left( \pi_{\Delta,e} \right) \subseteq \Hom_R(R^{1/p^e}, R)$ and all $a \in \ba^{\lceil t(p^e - 1) \rceil}$, we have that
\[
\phi( (aJ)^{1/p^e}) \subseteq J.
\]
A prime uniformly $(\Delta, \ba^t, F)$-compatible ideal is called a \emph{center of $F$-purity for $(R, \Delta, \ba^t)$}.
\end{definition}

\begin{remark}
In \cite{SchwedeCentersOfFPurity}, the author actually dealt with triples of the form $(R, \Delta, \ba_{\bullet})$ where $\ba_{\bullet}$ is a graded system of ideals (that is, $\ba_i \cdot \ba_j \subseteq \ba_{i+j}$).  For simplicity, we won't work with graded systems of ideals, although none of the results are more difficult in that generality.
\end{remark}

\begin{remark}
An ideal $J$ is uniformly $(\Delta, \ba^t, F)$-compatible if and only if for all
\[
\phi \in \Image\left( \pi_{\Delta,e} \right) \cdot \left(\ba^{\lceil t(p^e - 1) \rceil}\right)^{1/p^e},
\]
we have that
\[
\phi(J^{1/p^e}) \subseteq J.
\]
This can be seen because maps of the form $\phi \cdot a^{1/p^e}$ for $\phi \in \Image\left( \pi_{\Delta,e} \right)$ and $a \in \ba^{\lceil t(p^e - 1) \rceil}$ generate $\Image\left( \pi_{\Delta,e} \right) \cdot \left(\ba^{\lceil t(p^e - 1) \rceil}\right)^{1/p^e}$ even as an $R$-module.  In other words, the definition of uniformly $F$-compatible ideals is the same under the paradigm of ``\emph{local} strong $F$-regularity/ \emph{local} sharp $F$-purity'' as it is under the paradigm of ``strong $F$-regularity / sharp $F$-purity''
\end{remark}

\begin{remark}
\label{RemErrorCorrection}
We now point out a misstatement in the current author's paper \cite[Corollary 4.6]{SchwedeCentersOfFPurity}.  In that paper, the author claimed that in the case of an $F$-finite normal domain $R$, a triple $(R, \Delta, \ba^t)$ is strongly $F$-regular if and only if $(R, \Delta, \ba^t)$ has no proper nontrivial centers of $F$-purity.  To prove this, the author showed (correctly) that a minimal prime of the non-strongly $F$-regular locus\footnote{The non-strongly $F$-regular locus still makes sense using the old definition of a strongly $F$-regular pair or triple.  It is the set of primes $Q \in \Spec R$ such that $(R_Q, \Delta|_{\Spec R_Q}, \ba^t_Q)$ is not strongly $F$-regular.  This is a closed set because if $(R, \Delta, \ba^t)$ is strongly $F$-regular at $Q \in \Spec R$, then it is also strongly $F$-regular (and locally strongly $F$-regular) in a neighborhood of $Q$.} was a center of $F$-purity.

Unfortunately, the non-strongly $F$-regular locus being empty is equivalent to a triple $(R, \Delta, \ba^t)$ being \emph{locally} strongly $F$-regular (which we expect is a different condition than being strongly $F$-regular).
Therefore, a correct statement, in the case of an $F$-finite normal domain, would be one of the following.
\begin{itemize}
 \item[(1)]  A triple $(R, \Delta, \ba^t)$ has empty non-strongly $F$-regular locus if and only if $(R, \Delta, \ba^t)$ has no proper non-trivial centers of $F$-purity.
 \item[(2)]  A triple $(R, \Delta, \ba^t)$ is locally strongly $F$-regular if and only if $(R, \Delta, \ba^t)$ has no proper non-trivial centers of $F$-purity.
 \item[(3)]  The big test ideal $\tau_b(R; \Delta, \ba^t) = R$ if and only if $(R, \Delta, \ba^t)$ has no proper non-trivial centers of $F$-purity.
\end{itemize}

On the other hand, the only place where \cite[Corollary 4.6]{SchwedeCentersOfFPurity} was applied in the paper \cite{SchwedeCentersOfFPurity}, was in a case where $\Delta = 0$ and $\ba = R$, see \cite[Corollary 7.8]{SchwedeCentersOfFPurity}.
\end{remark}

Perhaps more importantly, all the results of \cite{SchwedeCentersOfFPurity} for sharply $F$-pure triples extend to ``locally sharply $F$-pure triples''.  For the most part, these generalizations can be accomplished by reducing to the local case where the two notions of strong $F$-regularity (respectively, sharp $F$-purity) agree.  However, the most fundamental such result is the following and we prove it directly:

\begin{proposition}\cite[Corollary 3.3]{SchwedeCentersOfFPurity}
If $(R, \Delta, \ba^t)$ is locally sharply $F$-pure and $I$ is uniformly $(\Delta, \ba^t, F)$-compatible, then $R/I$ is also $F$-pure.
\end{proposition}
\begin{proof}
Choose $\phi \in \Image\left( \pi_{\Delta,e} \right) \cdot \left(\ba^{\lceil t(p^e - 1) \rceil}\right)^{1/p^e}$ such that $\phi(1) = 1$.  Consider the diagram:
\[
\xymatrix@R=18pt{
I^{1/p^e} \ar@{^{(}->}[d] \ar[r]^{\phi|_{I^{1/p^e}}} & I \ar@{^{(}->}[d]\\
R^{1/p^e} \ar@{->>}[d] \ar[r]^{\phi} & R \ar@{->>}[d]\\
(R/I)^{1/p^e} \ar[r]^{\overline \phi} & R/I.
}
\]
Since $\phi(1) = 1$, we also see that $\overline \phi(1) = 1$.  This completes the proof.
\end{proof}

\subsection{Links with Kawamata log terminal singularities}

We finally note that the new Definition of strongly $F$-regular, provided in this paper, is the notion that corresponds to Kawamata log terminal singularities when $R$ is not necessarily local.

\begin{proposition} \cite{HaraRatImpliesFRat}, \cite[Theorem 3.2]{TakagiInterpretationOfMultiplierIdeals}, \cite[Theorem 6.8]{HaraYoshidaGeneralizationOfTightClosure}, \cite[Theorem 2.5]{TakagiPLTAdjoint}
\label{PropKLTImpliesLocallyFReg}
Let $(R, \Delta, \ba^t)$ be a triple that was reduced to characteristic $p \gg 0$ from a characteristic zero Kawamata log terminal triple\footnote{Note that this implies that $K_{R_0} + \Delta_0$ was $\bQ$-Cartier in characteristic zero.} (reduced with a log resolution, etc).  Then $(R, \Delta, \ba^t)$ is locally strongly $F$-regular.
\end{proposition}
\begin{proof}
This statement has typically been stated in the case that $R$ is a local ring (for example, it follows from \cite[Theorem 2.5]{TakagiPLTAdjoint}, also see \cite[Theorem 3.2]{TakagiInterpretationOfMultiplierIdeals} and \cite{HaraYoshidaGeneralizationOfTightClosure}).  A priori, if you change the local ring, you might need to also change the particular characteristic $p \gg 0$ you are working in.  However, the key injectivity needed to prove these results, holds at every local ring of a ring reduced to characteristic $p \gg 0$, see \cite[Subsection 4.4, 4.5]{HaraRatImpliesFRat}.  In particular, if $(R, \Delta, \ba^t)$ is as stated above, then after localizing at each prime $Q \in \Spec R$, $(R, \Delta|_{\Spec R_Q}, \ba_Q^t)$ is strongly $F$-regular.  Proposition \ref{PropKLTImpliesLocallyFReg} follows.
\end{proof}
\vskip 11pt
\hskip -11pt
{\it Acknowledgements: }  The author would like to thank Shunsuke Takagi and the referee for a careful reading and many useful comments on a previous draft of this paper.  He would also like to thank Manuel Blickle, Craig Huneke, Shunsuke Takagi, and Wenliang Zhang for several useful discussions.


\providecommand{\bysame}{\leavevmode\hbox to3em{\hrulefill}\thinspace}
\providecommand{\MR}{\relax\ifhmode\unskip\space\fi MR}
\providecommand{\MRhref}[2]{%
  \href{http://www.ams.org/mathscinet-getitem?mr=#1}{#2}
}
\providecommand{\href}[2]{#2}

\end{document}